\documentclass[11pt]{amsart}

\usepackage[utf8]{inputenc}
\usepackage[T1]{fontenc}
\usepackage{amsmath}
\usepackage{amssymb}
\usepackage{amscd}
\usepackage{color}
\usepackage{hyperref}
\usepackage{mathrsfs}
\usepackage{eucal}
\usepackage{upgreek}
\usepackage[makeroom]{cancel}
\usepackage[normalem]{ulem}
\usepackage{array}
\usepackage{verbatim}
\usepackage{mathtools}
\usepackage{stmaryrd}

\usepackage{enumitem}

\usepackage{xy}
\xyoption{all}
\usepackage{tikz-cd}

\newcommand{\nc}{\newcommand}

\nc{\md}{\operatorname{-}}

\nc{\CC}{{\mathbb{C}}}
\nc{\DD}{{\mathbb{D}}}
\nc{\LL}{{\mathbb{L}}}
\nc{\RR}{{\mathbb{R}}}
\renewcommand{\P}{{\mathbb{P}}}
\nc{\OO}{{\mathbb{O}}}

\nc{\QQ}{{\mathbb{Q}}}
\nc{\ZZ}{{\mathbb{Z}}}
\nc{\Z}{{\mathbb{Z}}}

\nc{\cA}{{\mathcal{A}}}
\nc{\cB}{{\mathcal{B}}}
\nc{\cC}{{\mathcal{C}}}
\nc{\cD}{{\mathcal{D}}}
\nc{\cE}{{\mathcal{E}}}
\nc{\cF}{{\mathcal{F}}}
\nc{\cG}{{\mathcal{G}}}
\nc{\cH}{{\mathcal{H}}}
\nc{\cI}{{\mathcal{I}}}
\nc{\cJ}{{\mathcal{J}}}
\nc{\cK}{{\mathcal{K}}}
\nc{\cL}{{\mathcal{L}}}
\nc{\cM}{{\mathcal{M}}}
\nc{\cN}{{\mathcal{N}}}
\nc{\cO}{{\mathcal{O}}}
\nc{\cP}{{\mathcal{P}}}
\nc{\cQ}{{\mathcal{Q}}}
\nc{\cR}{{\mathcal{R}}}
\nc{\cS}{{\mathcal{S}}}
\nc{\cT}{{\mathcal{T}}}
\nc{\cU}{{\mathcal{U}}}
\nc{\cV}{{\mathcal{V}}}
\nc{\cW}{{\mathcal{W}}}
\nc{\cX}{{\mathcal{X}}}
\nc{\cY}{{\mathcal{Y}}}
\nc{\cZ}{{\mathcal{Z}}}

\nc{\rc}{{\mathrm{c}}}
\nc{\rd}{{\mathrm{d}}}
\nc{\rf}{{\mathrm{f}}}
\nc{\rh}{{\mathrm{h}}}
\nc{\rrm}{{\mathrm{m}}}
\nc{\rs}{{\mathrm{s}}}
\nc{\rch}{{\mathrm{ch}}}
\nc{\rtd}{{\mathrm{td}}}

\nc{\rA}{{\mathrm{A}}}
\nc{\rB}{{\mathrm{B}}}
\nc{\rC}{{\mathrm{C}}}
\nc{\rD}{{\mathrm{D}}}
\nc{\rE}{{\mathrm{E}}}
\nc{\rF}{{\mathrm{F}}}
\nc{\rG}{{\mathrm{G}}}
\nc{\rH}{{\mathrm{H}}}
\nc{\rI}{{\mathrm{I}}}
\nc{\rJ}{{\mathrm{J}}}
\nc{\rK}{{\mathrm{K}}}
\nc{\rL}{{\mathrm{L}}}
\nc{\rM}{{\mathrm{M}}}
\nc{\rN}{{\mathrm{N}}}
\nc{\rO}{{\mathrm{O}}}
\nc{\rP}{{\mathrm{P}}}
\nc{\rQ}{{\mathrm{Q}}}
\nc{\rR}{{\mathrm{R}}}
\nc{\rS}{{\mathrm{S}}}
\nc{\rT}{{\mathrm{T}}}
\nc{\rU}{{\mathrm{U}}}
\nc{\rV}{{\mathrm{V}}}
\nc{\rW}{{\mathrm{W}}}
\nc{\rX}{{\mathrm{X}}}
\nc{\rY}{{\mathrm{Y}}}
\nc{\rZ}{{\mathrm{Z}}}

\nc{\bA}{{\mathbf{A}}}
\nc{\bB}{{\mathbf{B}}}
\nc{\bC}{{\mathbf{C}}}
\nc{\bD}{{\mathbf{D}}}
\nc{\bE}{{\mathbf{E}}}
\nc{\bF}{{\mathbf{F}}}
\nc{\bG}{{\mathbf{G}}}
\nc{\bH}{{\mathbf{H}}}
\nc{\bI}{{\mathbf{I}}}
\nc{\bJ}{{\mathbf{J}}}
\nc{\bK}{{\mathbf{K}}}
\nc{\bL}{{\mathbf{L}}}
\nc{\bM}{{\mathbf{M}}}
\nc{\bN}{{\mathbf{N}}}
\nc{\bO}{{\mathbf{O}}}
\nc{\bP}{{\mathbf{P}}}
\nc{\bQ}{{\mathbf{Q}}}
\nc{\bR}{{\mathbf{R}}}
\nc{\bS}{{\mathbf{S}}}
\nc{\bT}{{\mathbf{T}}}
\nc{\bU}{{\mathbf{U}}}
\nc{\bV}{{\mathbf{V}}}
\nc{\bW}{{\mathbf{W}}}
\nc{\bX}{{\mathbf{X}}}
\nc{\bY}{{\mathbf{Y}}}
\nc{\bZ}{{\mathbf{Z}}}

\nc{\ba}{{\mathbf{a}}}
\nc{\bb}{{\mathbf{b}}}
\nc{\bc}{{\mathbf{c}}}
\nc{\bd}{{\mathbf{d}}}
\nc{\be}{{\mathbf{e}}}
\nc{\bg}{{\mathbf{g}}}
\nc{\bh}{{\mathbf{h}}}
\nc{\bi}{{\mathbf{i}}}
\nc{\bj}{{\mathbf{j}}}
\nc{\bk}{{\mathbf{k}}}
\nc{\bl}{{\mathbf{l}}}
\nc{\bm}{{\mathbf{m}}}
\nc{\bn}{{\mathbf{n}}}
\nc{\bo}{{\mathbf{o}}}
\nc{\bp}{{\mathbf{p}}}
\nc{\bq}{{\mathbf{q}}}
\nc{\br}{{\mathbf{r}}}
\nc{\bs}{{\mathbf{s}}}
\nc{\bt}{{\mathbf{t}}}
\nc{\bu}{{\mathbf{u}}}
\nc{\bv}{{\mathbf{v}}}
\nc{\bw}{{\mathbf{w}}}
\nc{\bx}{{\mathbf{x}}}
\nc{\by}{{\mathbf{y}}}
\nc{\bz}{{\mathbf{z}}}

\nc{\fA}{{\mathfrak{A}}}
\nc{\fB}{{\mathfrak{B}}}
\nc{\fC}{{\mathfrak{C}}}
\nc{\fD}{{\mathfrak{D}}}
\nc{\fE}{{\mathfrak{E}}}
\nc{\fF}{{\mathfrak{F}}}
\nc{\fG}{{\mathfrak{G}}}
\nc{\fH}{{\mathfrak{H}}}
\nc{\fI}{{\mathfrak{I}}}
\nc{\fJ}{{\mathfrak{J}}}
\nc{\fK}{{\mathfrak{K}}}
\nc{\fL}{{\mathfrak{L}}}
\nc{\fM}{{\mathfrak{M}}}
\nc{\fN}{{\mathfrak{N}}}
\nc{\fO}{{\mathfrak{O}}}
\nc{\fP}{{\mathfrak{P}}}
\nc{\fQ}{{\mathfrak{Q}}}
\nc{\fR}{{\mathfrak{R}}}
\nc{\fS}{{\mathfrak{S}}}
\nc{\fT}{{\mathfrak{T}}}
\nc{\fU}{{\mathfrak{U}}}
\nc{\fV}{{\mathfrak{V}}}
\nc{\fW}{{\mathfrak{W}}}
\nc{\fX}{{\mathfrak{X}}}
\nc{\fY}{{\mathfrak{Y}}}
\nc{\fZ}{{\mathfrak{Z}}}

\nc{\fa}{{\mathfrak{a}}}
\nc{\fb}{{\mathfrak{b}}}
\nc{\fc}{{\mathfrak{c}}}
\nc{\fd}{{\mathfrak{d}}}
\nc{\fe}{{\mathfrak{e}}}
\nc{\ff}{{\mathfrak{f}}}
\nc{\fg}{{\mathfrak{g}}}
\nc{\fh}{{\mathfrak{h}}}
\nc{\fj}{{\mathfrak{j}}}
\nc{\fk}{{\mathfrak{k}}}
\nc{\fl}{{\mathfrak{l}}}
\nc{\fm}{{\mathfrak{m}}}
\nc{\fn}{{\mathfrak{n}}}
\nc{\fo}{{\mathfrak{o}}}
\nc{\fp}{{\mathfrak{p}}}
\nc{\fq}{{\mathfrak{q}}}
\nc{\fr}{{\mathfrak{r}}}
\nc{\fs}{{\mathfrak{s}}}
\nc{\ft}{{\mathfrak{t}}}
\nc{\fu}{{\mathfrak{u}}}
\nc{\fv}{{\mathfrak{v}}}
\nc{\fw}{{\mathfrak{w}}}
\nc{\fx}{{\mathfrak{x}}}
\nc{\fy}{{\mathfrak{y}}}
\nc{\fz}{{\mathfrak{z}}}

\nc{\sA}{{\mathsf{A}}}
\nc{\sB}{{\mathsf{B}}}
\nc{\sC}{{\mathsf{C}}}
\nc{\sD}{{\mathsf{D}}}
\nc{\sE}{{\mathsf{E}}}
\nc{\sF}{{\mathsf{F}}}
\nc{\sG}{{\mathsf{G}}}
\nc{\sH}{{\mathsf{H}}}
\nc{\sI}{{\mathsf{I}}}
\nc{\sJ}{{\mathsf{J}}}
\nc{\sK}{{\mathsf{K}}}
\nc{\sL}{{\mathsf{L}}}
\nc{\sM}{{\mathsf{M}}}
\nc{\sN}{{\mathsf{N}}}
\nc{\sO}{{\mathsf{O}}}
\nc{\sP}{{\mathsf{P}}}
\nc{\sQ}{{\mathsf{Q}}}
\nc{\sR}{{\mathsf{R}}}
\nc{\sS}{{\mathsf{S}}}
\nc{\sT}{{\mathsf{T}}}
\nc{\sU}{{\mathsf{U}}}
\nc{\sV}{{\mathsf{V}}}
\nc{\sW}{{\mathsf{W}}}
\nc{\sX}{{\mathsf{X}}}
\nc{\sY}{{\mathsf{Y}}}
\nc{\sZ}{{\mathsf{Z}}}

\nc{\sa}{{\mathsf{a}}}
\nc{\sd}{{\mathsf{d}}}
\nc{\se}{{\mathsf{e}}}
\nc{\sg}{{\mathsf{g}}}
\nc{\sh}{{\mathsf{h}}}
\nc{\si}{{\mathsf{i}}}
\nc{\sj}{{\mathsf{j}}}
\nc{\sk}{{\mathsf{k}}}
\nc{\sm}{{\mathsf{m}}}
\nc{\sn}{{\mathsf{n}}}
\nc{\so}{{\mathsf{o}}}
\nc{\sq}{{\mathsf{q}}}
\nc{\sr}{{\mathsf{r}}}
\nc{\st}{{\mathsf{t}}}
\nc{\su}{{\mathsf{u}}}
\nc{\sv}{{\mathsf{v}}}
\nc{\sw}{{\mathsf{w}}}
\nc{\sx}{{\mathsf{x}}}
\nc{\sy}{{\mathsf{y}}}
\nc{\sz}{{\mathsf{z}}}

\nc{\oA}{{\overline{A}}}
\nc{\oB}{{\overline{B}}}
\nc{\oC}{{\overline{C}}}
\nc{\oD}{{\overline{D}}}
\nc{\oE}{{\overline{E}}}
\nc{\oF}{{\overline{F}}}
\nc{\oG}{{\overline{G}}}
\nc{\oH}{{\overline{H}}}
\nc{\oI}{{\overline{I}}}
\nc{\oJ}{{\overline{J}}}
\nc{\oK}{{\overline{K}}}
\nc{\oL}{{\overline{L}}}
\nc{\oM}{{\overline{M}}}
\nc{\oN}{{\overline{N}}}
\nc{\oO}{{\overline{O}}}
\nc{\oP}{{\overline{P}}}
\nc{\oQ}{{\overline{Q}}}
\nc{\oR}{{\overline{R}}}
\nc{\oS}{{\overline{S}}}
\nc{\oT}{{\overline{T}}}
\nc{\oU}{{\overline{U}}}
\nc{\oV}{{\overline{V}}}
\nc{\oW}{{\overline{W}}}
\nc{\oX}{{\overline{X}}}
\nc{\oY}{{\overline{Y}}}
\nc{\oZ}{{\overline{Z}}}

\nc{\oa}{{\overline{a}}}
\nc{\ob}{{\overline{b}}}
\nc{\oc}{{\overline{c}}}
\nc{\od}{{\overline{d}}}
\nc{\of}{{\overline{f}}}
\nc{\og}{{\overline{g}}}
\nc{\oh}{{\overline{h}}}
\nc{\oi}{{\overline{i}}}
\nc{\oj}{{\overline{j}}}
\nc{\ok}{{\overline{k}}}
\nc{\ol}{{\overline{l}}}
\nc{\om}{{\overline{m}}}
\nc{\on}{{\overline{n}}}
\nc{\oo}{{\overline{o}}}
\nc{\op}{{\overline{p}}}
\nc{\oq}{{\overline{q}}}
\nc{\os}{{\overline{s}}}
\nc{\ot}{{\overline{t}}}
\nc{\ou}{{\overline{u}}}
\nc{\ov}{{\overline{v}}}
\nc{\ow}{{\overline{w}}}
\nc{\ox}{{\overline{x}}}
\nc{\oy}{{\overline{y}}}
\nc{\oz}{{\overline{z}}}

\nc{\tA}{{\tilde{A}}}
\nc{\tB}{{\tilde{B}}}
\nc{\tC}{{\tilde{C}}}
\nc{\tD}{{\tilde{D}}}
\nc{\tE}{{\tilde{E}}}
\nc{\tF}{{\tilde{F}}}
\nc{\tG}{{\tilde{G}}}
\nc{\tH}{{\tilde{H}}}
\nc{\tI}{{\tilde{I}}}
\nc{\tJ}{{\tilde{J}}}
\nc{\tK}{{\tilde{K}}}
\nc{\tL}{{\tilde{L}}}
\nc{\tM}{{\tilde{M}}}
\nc{\tN}{{\tilde{N}}}
\nc{\tO}{{\tilde{O}}}
\nc{\tP}{{\tilde{P}}}
\nc{\tQ}{{\tilde{Q}}}
\nc{\tR}{{\tilde{R}}}
\nc{\tS}{{\tilde{S}}}
\nc{\tT}{{\tilde{T}}}
\nc{\tU}{{\tilde{U}}}
\nc{\tV}{{\tilde{V}}}
\nc{\tW}{{\tilde{W}}}
\nc{\tX}{{\tilde{X}}}
\nc{\tY}{{\tilde{Y}}}
\nc{\tZ}{{\tilde{Z}}}

\nc{\tfD}{{\tilde{\fD}}}
\nc{\tcA}{{\tilde{\cA}}}
\nc{\tcB}{{\tilde{\cB}}}
\nc{\tcC}{{\tilde{\cC}}}
\nc{\tcD}{{\tilde{\cD}}}
\nc{\tcE}{{\tilde{\cE}}}
\nc{\tcF}{{\tilde{\cF}}}
\nc{\tcM}{{\tilde{\cM}}}
\nc{\tcP}{{\tilde{\cP}}}
\nc{\tcT}{{\tilde{\cT}}}

\nc{\ta}{{\tilde{a}}}
\nc{\tb}{{\tilde{b}}}
\nc{\tc}{{\tilde{c}}}
\nc{\td}{{\tilde{d}}}
\nc{\te}{{\tilde{e}}}
\nc{\tf}{{\tilde{f}}}
\nc{\tg}{{\tilde{g}}}
\nc{\ti}{{\tilde{\imath}}}
\nc{\tj}{{\tilde{j}}}
\nc{\tk}{{\tilde{k}}}
\nc{\tl}{{\tilde{l}}}
\nc{\tm}{{\tilde{m}}}
\nc{\tn}{{\tilde{n}}}
\nc{\tp}{{\tilde{p}}}
\nc{\tq}{{\tilde{q}}}
\nc{\tr}{{\tilde{r}}}
\nc{\ts}{{\tilde{s}}}
\nc{\tu}{{\tilde{u}}}
\nc{\tv}{{\tilde{v}}}
\nc{\tw}{{\tilde{w}}}
\nc{\tx}{{\tilde{x}}}
\nc{\ty}{{\tilde{y}}}
\nc{\tz}{{\tilde{z}}}

\nc{\hA}{{\hat{A}}}
\nc{\hB}{{\hat{B}}}
\nc{\hC}{{\hat{C}}}
\nc{\hD}{{\hat{D}}}
\nc{\hE}{{\hat{E}}}
\nc{\hF}{{\hat{F}}}
\nc{\hG}{{\hat{G}}}
\nc{\hH}{{\hat{H}}}
\nc{\hI}{{\hat{I}}}
\nc{\hJ}{{\hat{J}}}
\nc{\hK}{{\hat{K}}}
\nc{\hL}{{\hat{L}}}
\nc{\hM}{{\hat{M}}}
\nc{\hN}{{\hat{N}}}
\nc{\hO}{{\hat{O}}}
\nc{\hP}{{\hat{P}}}
\nc{\hQ}{{\hat{Q}}}
\nc{\hR}{{\hat{R}}}
\nc{\hS}{{\hat{S}}}
\nc{\hT}{{\hat{T}}}
\nc{\hU}{{\hat{U}}}
\nc{\hV}{{\hat{V}}}
\nc{\hW}{{\hat{W}}}
\nc{\hX}{{\widehat{X}}}
\nc{\hY}{{\hat{Y}}}
\nc{\hZ}{{\hat{Z}}}

\nc{\ha}{{\hat{a}}}
\nc{\hb}{{\hat{b}}}
\nc{\hc}{{\hat{c}}}
\nc{\hd}{{\hat{d}}}
\nc{\he}{{\hat{e}}}
\nc{\hg}{{\hat{g}}}
\nc{\hh}{{\hat{h}}}
\nc{\hi}{{\hat{i}}}
\nc{\hj}{{\hat{j}}}
\nc{\hk}{{\hat{k}}}
\nc{\hl}{{\hat{l}}}
\nc{\hm}{{\hat{m}}}
\nc{\hn}{{\hat{n}}}
\nc{\ho}{{\hat{o}}}
\nc{\hp}{{\hat{p}}}
\nc{\hq}{{\hat{q}}}
\nc{\hr}{{\hat{r}}}
\nc{\hs}{{\hat{s}}}
\nc{\hu}{{\hat{u}}}
\nc{\hv}{{\hat{v}}}
\nc{\hw}{{\hat{w}}}
\nc{\hx}{{\hat{x}}}
\nc{\hy}{{\hat{y}}}
\nc{\hz}{{\hat{z}}}

\nc{\hcC}{{\widehat{\cC}}}
\nc{\hcT}{{\widehat{\cT}}}

\nc{\eps}{\upepsilon}
\nc{\lan}{\big\langle}
\nc{\ran}{\big\rangle}
\nc{\kk}{{\Bbbk}}
\nc{\io}{\upiota}
\nc{\Kr}{\mathsf{Kr}}
\nc{\cKr}{\mathcal{K}\!\mathit{r}}

\nc{\Dm}{\bD^{-}}
\nc{\Db}{\bD^{\mathrm{b}}}
\nc{\Dbc}{\bD^{\mathrm{b}}_{\mathrm{c}}}
\nc{\Dp}{\bD^{\mathrm{perf}}}
\nc{\Dperf}{\bD^{\mathrm{perf}}}
\nc{\Dqc}{\bD_{\mathrm{qc}}}
\nc{\Du}{\bD}
\nc{\Dsing}{\bD^{\mathrm{sg}}}
\nc{\Dg}{\bD^{\mathrm{sg}}}

\def\ol{\overline}

\nc{\Rn}{\rR_{\mathrm{node}}}
\nc{\Cn}{\cC_{\mathrm{node}}}
\nc{\Dfd}[1]{\bD_{\mathrm{fd}}(#1)}

\def\bw#1#2{\textstyle{\bigwedge\hskip-0.9mm^{#1}}\hskip0.2mm{#2}}

\nc{\xrightiso}[1]{ \xrightarrow[{\ \raisebox{0.5ex}[0ex][0ex]{$\sim$}\ }]{#1} }

\nc{\thick}{\mathbf{thick}}

\DeclareMathOperator{\Hom}{\mathrm{Hom}}

\DeclareMathOperator{\Spec}{\mathrm{Spec}}

\DeclareMathOperator{\Pic}{\mathrm{Pic}}
\DeclareMathOperator{\Cl}{\mathrm{Cl}}

\DeclareMathOperator{\CH}{\mathrm{CH}}

\DeclareMathOperator{\Ker}{\mathrm{Ker}}
\DeclareMathOperator{\Coker}{\mathrm{Coker}}

\DeclareMathOperator{\Cone}{\mathrm{Cone}}

\DeclareMathOperator{\rank}{\mathrm{rk}}
\DeclareMathOperator{\codim}{\mathrm{codim}}

\theoremstyle{plain}

\newtheorem{theorem}{Theorem}[section]
\newtheorem{conjecture}[theorem]{Conjecture}

\newtheorem{lemma}[theorem]{Lemma}
\newtheorem{proposition}[theorem]{Proposition}
\newtheorem{corollary}[theorem]{Corollary}

\theoremstyle{definition}

\newtheorem{definition}[theorem]{Definition}

\newtheorem{example}[theorem]{Example}

\theoremstyle{remark}

\newcommand{\corr}[1]{#1}

\title{Homological Bondal--Orlov localization conjecture for rational singularities}

\author[M.~Mauri]{Mirko Mauri}
\address{Institute of Science and Technology Austria, 3400 Klosterneuburg, Austria.}
\email{mirko.mauri@ist.ac.at}

\author[E.~Shinder]{Evgeny Shinder}
\address{
School of Mathematics and Statistics, University of Sheffield,
Hounsfield Road, S3 7RH, UK, and
Hausdorff Center for Mathematics
at the University of Bonn, Endenicher Allee 60, 53115.
}
\email{eugene.shinder@gmail.com}

\begin{document}
\maketitle
\begin{abstract}
Given a resolution of rational singularities
$\pi\colon \tX \to X$ over a field of characteristic zero we use a Hodge-theoretic
argument to prove that the 
image of the functor
$\bR\pi_*\colon \Db(\tX) \to \Db(X)$
between bounded derived categories
of coherent sheaves
generates $\Db(X)$ as a triangulated category.
This 
gives a weak version of the Bondal--Orlov localization conjecture \cite{BO02},
answering a 
question from \cite{PS18}.
The same result is established
more generally 
for proper 
(\corr{not necessarily} birational) morphisms $\pi\colon \tX \to X$, with $\tX$ smooth, satisfying $\bR\pi_*(\cO_\tX) = \cO_X$.
\end{abstract}

\section{Introduction}

Operations of the Minimal Model Program
often correspond
to operations on the  derived categories of coherent sheaves.
The case of smooth varieties \corr{is understood relatively well. In
particular, projective} bundles, blow-ups
and standard flips
correspond to semiorthogonal decompositions
of derived categories, see the 2002 ICM report
by Bondal and Orlov \cite{BO02} and references therein.
Furthermore, the K-equivalence conjecture
of Kawamata \cite{Kawamata-K-equivalence}, 
building on ideas of Bondal and Orlov predicts that K-equivalent varieties
should have equivalent derived categories.
In particular, a crepant
resolution of a singular variety
is conjecturally unique up to derived equivalence.
For a recent survey
of various aspects
of the interplay
between birational geometry
and derived categories
for smooth projective
varieties, see \cite{Kawamata-birational},
and for conjectural relationship to rationality problems,
see \cite{Kuznetsov-rationality}.

On the other hand, much less is known
for \corr{singular varieties. In particular, the relationship between properties
of the derived category of a singular
variety
and its resolution is still unclear.} Of course this relationship should depend
on the type of singularities of $X$. 
Recall that a variety
 $X$ over a field of characteristic zero has rational singularities
if
the derived pushforward $\bR\pi_*\cO_\tX$ coincides
with $\cO_X$ for 
some (hence, every) resolution $\pi\colon \tX \to X$.
One old major open question is the \emph{Bondal--Orlov localization conjecture}:

\begin{conjecture}
\cite[Section 5]{BO02},
\cite[Conjecture 1.9]{Ef20}
\label{conj:BO}
Let $\pi\colon \tX \to X$ be a resolution of rational singularities.
Then the functor
$
\bR\pi_*\colon \Db(\tX) \to \Db(X)$
between bounded
derived categories of coherent
sheaves is a Verdier localization,
that is the induced functor \[\ol{\bR\pi_*}\colon \Db(\tX) / \Ker(\bR\pi_*) \to \Db(X)\] is an equivalence.

\end{conjecture}

It is convenient to split
this statement into two parts:
\begin{enumerate}
    \item[(a)] the induced functor $\ol{\bR\pi_*}\colon \Db(\tX) / \Ker(\bR\pi_*) \to \Db(X)$ is fully faithful;
    \item[(b)] the functor $\bR\pi_*\colon \Db(\tX) \to \Db(X)$
    is essentially surjective.
\end{enumerate}

Validity of Conjecture \ref{conj:BO}
does not depend on 
the choice of a resolution
of $X$, at least in characteristic zero \cite[Lemma 2.31]{PS18}.
Understanding
this conjecture is essential for 
linking the Minimal Model Program to operations
on derived categories, as well as for understanding derived categories
of singular varieties. Indeed, 
a typical investigation of $\Db(X)$ proceeds by descending results
from $\Db(\tX)$ to $\Db(X)$, typically restricting to cases
when the localization conjecture holds, see e.g. \cite{KKS20, KSabs} for how this strategy can be implemented.

Recent progress on 
Conjecture \ref{conj:BO} 
includes \cite{Ef20} 
\cite{BKS18},
\cite{PS18}.
Notably it
holds for cones
over
projectively normal smooth Fano varieties
\cite{Ef20}
(as explained in \cite[Example 5.4, Corollary 5.6]{KSabs}),
all quotient singularities in characteristic zero \cite[Theorem 2.30]{PS18}
and singularities admitting
a resolution with $1$-dimensional fibers \cite[Theorem 2.14]{BKS18},
however in general it remains
wide open.

Our main result is the
weaker version
of (b) which answers the question
asked
in \cite[Introduction]{PS18}.
Let $\rG_0(X) := \rK_0(\Db(X))$
be the Grothendieck group; it is isomorphic
to the Grothendieck group of the abelian category of
coherent sheaves on $X$.
The group $\rG_0(X)$ is covariantly
functorial
for proper morphisms, hence plays the role
of a
K-theoretic version 
of the 
Borel--Moore homology, similarly to 
Chow groups.
By the `Homological Bondal--Orlov localization' we mean
the implications
of Conjecture
\ref{conj:BO} for $\rG_0(X)$
as explained in
\cite[\S 4 of the Introduction]{PS18}.
We prove the following:

\begin{theorem}\label{thm:main-intro}
Let $\pi\colon \tX \to X$
be a proper morphism of algebraic varieties
over a field $\kk$ of characteristic zero with $\tilde{X}$ smooth and satisfying
$\bR\pi_*(\cO_\tX) = \cO_X$.
Then the image of
$\bR\pi_*\colon \Db(\tX) \to \Db(X)$
generates 
$\Db(X)$ as \corr{a
triangulated category. In particular,}
the induced homomorphism
$
\bR\pi_*\colon \rG_0(\tX) \to \rG_0(X)
$
is surjective.
\end{theorem}

Here we say that a set of objects
$\cS \subset \Db(X)$ generates $\Db(X)$ 
as a triangulated
category if every
object of $\Db(X)$ can be obtained
from the objects of $\cS$ by 
iterating
cones and \corr{shifts. 
Importantly,} taking
direct summands is not required (this would make the statement much weaker, with no implications
about the map on $\rG_0$).

We prove this by reducing to $\kk = \CC$
and using an argument from
Hodge theory going back
to Steenbrink's work on Hodge
theory of singularities \cite{Steenbrink-mixed}
(see Lemma \ref{lem:Steenbrink-NEW}
for this step).
Theorem \ref{thm:main-intro} applies in particular to resolutions of rational singularities.
We note that this kind of surjectivity is specific to derived categories of coherent sheaves and $\rG_0(X)$; it is easy to find examples when it fails
for other kinds of Borel--Moore homology theories, such as Chow groups or singular homology, see Examples 
\ref{rem:homology1}
and \ref{rem:homology2}.

Let us explain why Theorem \ref{thm:main-intro}, 
as well as
the full
Bondal--Orlov localization Conjecture
\ref{conj:BO},
is a nontrivial and subtle
statement. 
The key reason why the condition $\bR\pi_*(\cO_\tX) = \cO_X$ is
relevant when studying derived
categories of coherent sheaves is the projection formula: for every $\cF \in \Db(X)$
\[
\bR\pi_* \bL\pi^*(\cF) \simeq \cF \otimes 
\bR\pi_*(\cO_{\tX}) = \cF,
\]
and this motivates at least part (b)
of Conjecture \ref{conj:BO}.
However, unless $\cF$ is a perfect complex,
$\bL\pi^*(\cF)$ is a complex
unbounded to the left, whereas the question
is about existence of a bounded complex on $\tX$. Thus,
a certain truncation of $\bL\pi^*(\cF)$
may be required,
but a truncation does not immediately 
yield the complex we need, see 
\cite[Lemma 7.4]{Kawamata-singular}
where it is shown that $\cF$
is a direct summand of the image
of a truncation of $\bL\pi^*(\cF)$.
In particular
we do not have a canonical lifting
of a complex $\cF \in \Db(X)$
or even its
class $[\cF] \in \rG_0(X)$ to 
$\tX$.

Our approach is 
indirect: using standard K-theory
arguments it suffices to consider
the structure
sheaves $\cO_Z \in \Db(X)$
of closed
subvarieties $Z \subset X$.
We construct a birational modification $\sigma \colon \tX' \to \tX$ such that the fibers of $\pi \sigma$, with their reduced scheme structure, have simple normal crossings, see Lemma \ref{lem:snc-fication}. Hodge theory is used to prove that
these
fibers have no higher
cohomology for the structure sheaf,
see Proposition \ref{prop:rational-acyclic}
and Lemma \ref{lem:Steenbrink-NEW}.
We call
the latter property of a morphism \emph{$\cO$-acyclicity}, and we use it to show that $\cO_Z$ lies in the image of $\bR\pi_*$ eventually up to a small correction term, namely that $\cO_{Z}$ coincides with $\bR\pi_*(\bR\sigma_*\cO_{(\pi \sigma)^{-1}(Z)_{\mathrm{red}}})$ away from a proper subset of $Z$, see Lemma \ref{lem:FZ}. Finally, we deduce Theorem \ref{thm:main-intro}
by induction on the dimension of $Z$,
see Proposition \ref{prop:acylic-surjective}.

From Theorem \ref{thm:main-intro} we can deduce
that part (b) of the Bondal--Orlov
localization 
(essential surjectivity)
already follows if we know part (a) 
(fully faithfulness):

\begin{corollary}\label{cor:cat-contraction}
Under assumptions of
Theorem \ref{thm:main-intro},
if 
$\ol{\bR\pi_*}\colon \Db(\tX)/\Ker(\bR\pi_*) \to \Db(X)$ 
is full (i.e. surjective on $\Hom$-spaces), then
$\ol{\bR\pi_*}$ is essentially surjective. In particular, if $\ol{\bR\pi_*}$ is
fully faithful,
then it
is a Verdier localization.
\end{corollary}

Thus, in the terminology
of \cite{KSabs, KS:hfd},
$\bR\pi_*$ is a Verdier localization
if and only if it is a so-called
categorical contraction, which is a priori a weaker statement meaning Verdier localization
up to direct summands.

\subsection*{Notation and conventions}

We work over a field $\kk$ of characteristic zero. Unless stated otherwise,
our varieties are assumed to be reduced and
quasi-projective over $\kk$.

\subsection*{Acknowledgements}
We thank 
Agnieszka Bodzenta-Skibińska, Paolo Cascini, 
Wahei Hara,
Sándor Kovács, Alexander Kuznetsov,
Mircea Musta\k{t}\u{a}, Nebojsa Pavic,
Pavel Sechin,
Michael Wemyss
for discussions and e-mail correspondence. We also thank the anonymous referee for the helpful comments.

\section{$\cO$-acyclic morphisms
and surjectivity of $\bR\pi_*$}

When working with derived categories
and rational singularities, the following
property is relevant.

\begin{definition}
\label{def:O-acyclic}
Let $\pi: \tX \to X$ be a surjective
proper 
morphism
with geometrically connected fibers. 
We say that $\pi$ is {\sf $\cO$-acyclic
at $x \in X$}
if 
$\rH^{>0}(\pi^{-1}(x)_{\mathrm{red}}, \cO) = 0$.

Here $x \in X$ is a scheme point
and $\pi^{-1}(x)$ is a
scheme of finite type over $\kk(x)$.
Note that because we assume fibers to be geometrically connected we have
\begin{equation}
\label{eq:fibers}
\rH^{*}(\pi^{-1}(x)_{\mathrm{red}}, \cO) = \kk[0].
\end{equation}

We say that $\pi$ is {\sf $\cO$-acyclic}
if it is acyclic at each point $x \in X$
and we say that
$\pi$ is {\sf $\cO$-acyclic
after a modification} if 
there exists
a proper
morphism $\sigma \colon \tX' \to \tX$
such that
$\pi \sigma$ is
$\cO$-acyclic.
\end{definition}

The following Lemma shows in particular
that to check that $\pi$
is $\cO$-acyclic it suffices to consider
fibers over 
closed points of $X$.
Here $\pi$ is the morphism as in Definition \ref{def:O-acyclic}.

\begin{lemma}\label{lem:closed-points}
$\pi$ is $\cO$-acyclic at $x \in X$
if and only if it is acyclic at
general closed 
points $x' \in \ol{\{x\}}$.
\end{lemma}
\begin{proof}
Let $Z = \ol{\{x\}}$
and $E$ be the reduced $\pi$-preimage
of $Z$. 
The result follows from \cite[Corollary 12.9]{Hartshorne} and the  generic flatness of the restricted morphism $\pi_Z\colon E \to Z$.
\end{proof}

In the case when $\pi$ is flat and 
all fibers are reduced, 
being $\cO$-acyclic is equivalent to $\bR\pi_*(\cO_{\tX}) = \cO_X$ by base change.
In general the relation between these two properties is quite subtle.

\begin{example}
Let
$\pi\colon \tX \to X$ be 
a resolution of singularities.
Then $\pi$ can be $\cO$-acyclic, 
but $\rR^{>0}\pi_*(\cO_\tX) \ne 0$, because
higher derived pushforwards involve
thickenings in a nontrivial way.

\begin{enumerate}
    \item[(1)] 
     The map $\pi$ can be chosen as a log resolution of a non-rational surface singularity whose exceptional locus $E$ is a tree of rational curves. For instance, consider the Brieskorn singularity $X$ given by the equation $ x^2+y^3+z^7=0$. The exceptional curve of its minimal resolution (the weighted blow-up with weights $(3,2,1)$) is the cuspidal cubic $C$ given by  $x^2+y^3=0$ in $\bP(3,2,1)$,
     and so \corr{$X$ does not have rational singularities} by Artin's rationality criterion since $\rH^1(C, \cO)=\kk$, see \cite[Proposition 1]{Artin1966}
     or Example \ref{ex:Artin} below. 
     On the other hand, by blowing-up the minimal resolution three times, we obtain a log resolution $\pi \colon \tX \to X$ whose exceptional locus $E$ is a tree of rational curves: its dual graph is $D_4$; the rational curve in the middle has self intersection $-1$, and the other ones $-2$, $-3$, $-7$. Note that $\rH^{1}(E, \cO)=0$, which means that $\pi$ is $\cO$-acyclic.

\item[(2)] Consider the affine cone $X$
over a smooth projective variety $Y \subset \P^N$ with  $\rH^{>0}(Y, \cO)=0$ but $\rH^{>0}(Y, \cO(1))\neq 0$, see \cite[Proposition 3.13]{Kollar-singMMP}. For instance
one can take the surface suggested
by Jason Starr in \cite{MO-Starr}:
$Y \subset \P^2 \times \P^1$
is a divisor of bidegree $(d,1)$ with $d \ge 4$. 
Then the blow-up of the vertex $\pi\colon \tX \to X$ 
has reduced $\cO$-acyclic fibers (fibers consist
of a point or of $Y$) but $X$ is
not a rational singularity.
\end{enumerate}
\end{example}

In the opposite direction we have the following:

\begin{proposition}\label{prop:rational-acyclic}
If $\pi\colon \tX \to X$ is a proper morphism with $\tX$ smooth and 
satisfying
$\bR\pi_*(\cO_\tX) = \cO_X$, then $\pi$
is $\cO$-acyclic after a modification.
\end{proposition}

This is a remarkable result
as we  
are essentially 
able to make 
a very naive 
base change statement
for a morphism $\pi$
which is not flat.
Partial results of this kind
can be obtained by the formal functions theorem (see Example \ref{ex:Artin}), however the formal functions theorem 
alone is not sufficient to control all cohomology groups
of the fibers.
We begin with the following particular
case of Proposition \ref{prop:rational-acyclic} which
we prove using 
complex algebraic geometry.

\begin{lemma} 
\label{lem:Steenbrink-NEW}
Let $\pi\colon \tX \to X$
be a proper morphism satisfying
$\bR\pi_*(\cO_\tX) = \cO_X$. 
Let $x \in X$ be a closed point
and
assume that $E = \pi^{-1}(x)_{\mathrm{red}}$
has simple normal crossings,
then $\pi$ is $\cO$-acyclic at $x$.
\end{lemma}
\begin{proof}
When $\kk = \CC$, this result has been proved in \cite[Proposition 3.1]{DuBois81}
for resolutions of cone singularities, and in
\cite[Proposition 3.7 and \S 3.6]{Steenbrink-mixed}, \cite[Proposition 8.1.11.(ii) and 8.1.12]{Ishii2014} or \cite[Lemma 1.2]{Namikawa} for resolutions of rational singularities. We give a 
self-contained proof for completeness.

Extending the scalars,
we can assume
that $\kk$ is an 
algebraically closed field.
By the Lefschetz principle we can
further assume that $\kk = \CC$
and prove the result via Hodge theory. By the GAGA principle,
the derived pushforward $\bR\pi_*$
is the same when computed in the analytic or in the Zariski topology \cite[Expose XII Theorem 4.2]{SGA1}.
Similarly $\rH^{*}(E, \cO)$
coincides when 
computed in the analytic
or in the Zariski topology.

Let $(V, x) \subseteq (\CC^n, 0)$ be an affine neighbourhood of $x$ in $X$. The intersection of $V$ with a ball of radius $\epsilon$ centered at $0$, denoted $U \coloneqq \{ v \in V \subseteq \CC^n  | \, |v|<\epsilon\}$, is a Stein neighbourhood of $x$ in $X$. The preimage $\pi^{-1}(U)$ is a Euclidean neighbourhood of $E$, and the inclusion
$j\colon E \to \pi^{-1}(U)$
is a homotopy equivalence; see e.g.\  \cite[Proposition 1.6]{Durfee}.
We consider the following diagram
containing 
singular cohomology
and analytic sheaf cohomology
\[\xymatrix{
\rH^*(\pi^{-1}(U), \CC) \ar[d]_{j^*} \ar[r] & 
\rH^*(\pi^{-1}(U), \cO) \ar[d]^{j^*} \\
\rH^*(E, \CC) \ar[r] & 
\rH^*(E, \cO).\\
}
\]
Here the horizontal maps are induced
by sheaf inclusions $\CC \subset \cO$
and the vertical maps are restrictions. 
Since $j$ is a homotopy
equivalence, the left vertical map is an isomorphism.
By the work of Deligne \cite{DeligneIII}, the cohomology groups of the simple normal crossing variety $E$ carry a mixed Hodge structure such that $F^0 / F^1$ is canonically isomorphic
to $\rH^*(E, \cO)$ (see e.g.\ \cite[\S (1.5)]{Steenbrink-mixed}), hence the bottom horizontal map is surjective. Thus the right vertical map
is surjective as well.
Since $\bR\pi_*(\cO_\tX) = \cO_X$ and $U$ is Stein,
we have
$\rH^{>0}(\pi^{-1}(U), \cO) = 0$
and 
this implies
$\rH^{>0}(E, \cO) = 0$.
\end{proof}

We do not know if the 
conclusion
of Lemma \ref{lem:Steenbrink-NEW}
holds for the 
scheme preimage
itself, without taking reduced scheme structure. The assumption on the singularities of $E$ can be weakened to the requirement that $E$ has only Du Bois singularities, which grants the surjectivity of the morphism $\rH^*(E, \CC) \to
\rH^*(E, \cO)$; see \cite[(1.4)]{Kovacs2012} for details on Du Bois singularities. However, it is not clear whether the assumption on the singularities of $E$ can be removed at all, at least when $\pi$ is birational. 
This is true for resolution of rational singularities of dimension two or more generally admitting
a resolution with $1$-dimensional fibers by the following generalization of Artin's rationality criterion, see \cite[Proposition 1]{Artin1966}.

\begin{example} 
\label{ex:Artin}
Let $\pi\colon \tX \to X$
be a proper morphism satisfying
$\bR\pi_*(\cO_\tX) = \cO_X$. Let $x \in X$ be a schematic point, and 
$F \subset \pi^{-1}(x)$
any closed subscheme;
in particular we can take $F = \pi^{-1}(x)$ or 
$F = \pi^{-1}(x)_{\mathrm{red}}$. Suppose that $d=\dim(\pi^{-1}(x))>0$.
Then $\rH^d(F, \cO)=0$. 
Indeed, let $E_{k} \coloneqq \tX \times_X \Spec \cO_x/ \mathfrak{m}^{k}_x$, where $\mathfrak{m}_x$ is the maximal ideal of $x$.
For $k>0$ large enough,
there exists an epimorphism $\cO_{E_{k}} \to \cO_{F}$. 
For dimensional reasons, 
if
$\cK_k$ is the kernel of $\cO_{E_{k}} \to \cO_{F}$,
then
$\rH^{d+1}(E_{k},\cK_k) = 0$. 
Hence, $\rH^d(F, \cO)=0$ as soon as $\rH^d(E_{k}, \cO)=0$. But the latter group vanishes by the formal function theorem
and $\bR\pi_*(\cO_\tX) = \cO_X$:  \[0=(\bR^{d}\pi_*(\cO_\tX)_x)^{\wedge} = \lim_{\longleftarrow} \rH^d(E_{k}, \cO) \twoheadrightarrow \rH^d(E_{k}, \cO).\] 
Surjectivity of the last
map holds because
all transition maps
$\rH^d(E_{k+1}, \cO) \to
\rH^d(E_{k}, \cO)$
are surjective
by the cohomology vanishing argument above.
\end{example}

\medskip

The next result produces a desingularization of the fibers of a morphism. \corr{For completeness we reproduce a proof due to Wlodarczyk.}

\begin{lemma}\cite[Proposition 6.0.5]{WJ2016}
\label{lem:snc-fication}
Let $\pi\colon \tX \to X$
be a proper morphism with $\tX$ smooth.
Then there exists a morphism
$\sigma\colon \tX' \to \tX$
which is a composition of blow-ups along smooth centers
such that all 
fibers of
$\pi\sigma$, with reduced
scheme structure,
are simple normal crossing
varieties.
\end{lemma}
\begin{proof}
We will
construct a sequence of (compositions of) smooth 
blow-ups
\[
\tX_n \overset{\sigma_n}\longrightarrow \tX_{n-1} \longrightarrow \dots 
\longrightarrow
\tX_1 \overset{\sigma_1}\longrightarrow \tX_0 = \tX
\]
such that for all $k \le n$
the morphism $\pi \sigma_1 \cdots \sigma_k$
has snc fibers over an open subset $U_k \subset X$
with the complement
$X \setminus U_k$ of codimension at least $k+1$.
Then we can take $\tX' = \tX_n$ and $\sigma = \sigma_1 \cdots \sigma_n$ for $n = \dim(X)$.

The variety 
$\tX_0 = \tX$ satisfies
our assumptions
because by generic smoothness,
$\pi$ is smooth over a
dense open subset of $X$.
Assume that $\tX_{k-1}$ is constructed.
Let $Z = X \setminus U_{k-1}$.
By assumption $Z$ has codimension at least $k$.
If the codimension of $Z$
is strictly larger than $k$, 
we can set $\tX_k = \tX_{k-1}$.
Otherwise let $Z_1, \dots, Z_m$ be $k$-codimensional
irreducible components of $Z$.
Let $\sigma_k\colon \tX_k \to \tX_{k-1}$ be a 
sequence of smooth blow-ups which
provides a log-resolution  
of 
$(\pi \sigma_1 \cdots \sigma_{k-1})^{-1}(Z_1 \cup \dots \cup Z_m)$, 
that is 
we require that
$\sigma_k$ is an isomorphism
away from the preimages of
$Z_1, \dots, Z_m$, and the 
preimage 
of every $Z_i$
with respect to
$\pi \sigma_1 \cdots \sigma_k$
is a snc divisor.
A simple argument using generic smoothness
guarantees that there exists an open subset $U_i \subseteq Z_i$ such that for any $x \in U_i$ 
the fibers 
$(\pi \sigma_1 \cdots \sigma_k)^{-1}(x)$ 
also have
simple normal crossings, see e.g.\ \cite[Proposition 4.0.4]{WJ2016}.
Thus the set of points $x \in X$ where
$(\pi \sigma_1 \cdots \sigma_k)^{-1}(x)$
is not an snc variety is contained in a closed
subset  
of codimension at least $k+1$.
\end{proof}

\begin{proof}[Proof of Proposition \ref{prop:rational-acyclic}]
We take the modification
$\sigma$ from Lemma
\ref{lem:snc-fication}.
Since $\sigma$ is a birational modification of the smooth variety $\tX$, 
we have $\bR \sigma_*(\cO_{\tX'}) = \cO_{\tX}$,
hence $\bR (\pi\sigma)_*(\cO_{\tX'}) = \cO_X$.
We need to check that
$\pi\sigma$ is $\cO$-acyclic.
By Lemma \ref{lem:closed-points}
it suffices to prove the same
for closed points $x \in X$.
The result now follows from Lemma \ref{lem:Steenbrink-NEW}.
\end{proof}

The following two lemmas
are used in our inductive
proof of Theorem \ref{thm:main-intro}.

\begin{lemma}\label{lem:FZ}
Let $\pi\colon \tX \to X$
be $\cO$-acyclic
at $x \in X$.
Let $Z \subset X$ be the 
closure of $x$
and $E \subset \tX$
be the reduced preimage of $Z$.
Consider the complex
\[
\cF_Z^\bullet := \Cone(\cO_Z \to \bR\pi_*(\cO_E)).
\]
Then $\cF_Z^\bullet$ is supported on a proper
closed subset of $Z$.
\end{lemma}
\begin{proof}
Replacing $\pi$ by 
the restricted 
morphism
$\pi_Z\colon E \to Z$
does not affect the definition of $\cF_Z^\bullet$.
The fiber of $\cF_Z^\bullet$ at the generic point $x \in Z$ is isomorphic to
\begin{equation}
\label{eq:cFZ-generic}
\Cone(\kk(x)[0] \to \rH^*(E_x, \cO)),
\end{equation}
where $E_x$ is the generic fiber of $\pi_Z$.
By assumption $\pi_Z$
is $\cO$-acyclic
at the generic point $x \in Z$,
hence \eqref{eq:fibers} implies
that \eqref{eq:cFZ-generic} vanishes. In particular, 
$\cF_Z^\bullet$ vanishes
at the generic point $x \in Z$. Therefore $\cF_Z^\bullet$ is supported
on a proper closed subset of $Z$.
\end{proof}

\begin{lemma}\label{lem:devissage}
Let $\cD^m$ be the triangulated
subcategory of $\Db(X)$ consisting
of complexes acyclic away from a codimension $m$ subset. Then
$\cD^m$ is generated by
$\cD^{m+1}$
and all $\cO_Z$ for $Z \subset X$
with $Z$ integral, $\codim(Z) = m$.
In particular, $\Db(X)$
is generated by all structure
sheaves $\cO_Z$.
\end{lemma}
\begin{proof}
This is the standard topological filtration argument in algebraic K-theory going back to Grothendieck,
see e.g.\ \cite[Lemma 1.7]{Calabrese-Pirisi}.
We recall the argument.
Filtering complexes
by their cohomology sheaves, it suffices
to consider a coherent sheaf $\cF$ set-theoretically
supported
on a reduced
subscheme $Z \subset X$ of codimension $m$.
Each subquotient
$\cI_Z^n \cF / \cI_Z^{n+1} \cF$
is a $\cO_X / \cI_Z$--module,
hence it is
scheme-theoretically supported on $Z$
so we can assume $\cF = i_*(\cF_0)$,
where $i\colon Z \to X$ is the inclusion,
for a coherent sheaf $\cF_0$ on $Z$.
A simple argument using induction on the number
of irreducible components of $Z$ allows
us to assume that $Z$ is integral.
In this case since we assume $X$ to be 
quasiprojective, there is a sufficiently
ample line bundle $\cO(H)$
and a morphism
\[
\phi\colon \cO_Z^{\oplus n}(-H) \to \cF_0
\]
which is an isomorphism over generic point of $Z$.
In other words, $\Cone(\phi) \in \cD^{m+1}$
and we can assume $\cF_0$ is a line bundle $\cO_Z(-H)$
on $Z$.
The same argument using
\[
\Cone(\cO_Z(-H) \to \cO_Z) \simeq \cO_{Z \cap H} \in \cD^{m+1}
\]
reduces
the statement to the case
$\cF_0 = \cO_Z$ and we are done.
\end{proof}

\begin{proposition}\label{prop:acylic-surjective}
If $\pi$ is $\cO$-acyclic after a modification, then the image
$\bR\pi_*(\Db(\tX))$ generates
$\Db(X)$ as a triangulated category.
\end{proposition}
\begin{proof}
By definition there exists a 
birational morphism
$\sigma\colon \tX' \to \tX$
of smooth varieties such that
$\pi\sigma$ is $\cO$-acyclic.
Since
$\bR\pi_*(\Db(\tX))$
contains
$\bR(\pi\sigma)_*(\Db(\tX))$
it suffices to show 
that
$\bR(\pi\sigma)_*(\Db(\tX))$
generates
$\Db(X)$ as a triangulated category.

Let $\cT$ be the triangulated
subcategory of $\Db(X)$
generated by 
$\bR(\pi\sigma)_*(\Db(\tX))$.
We need to show
that $\cT = \Db(X)$.
Let $\cD^m$ be the subcategory of $\Db(X)$
consisting of complexes acyclic away from a codimension $m$ subset.
We check by the descending
induction on $m$ that $\cD^m \subset \cT$ for all 
$0 \le m \le \dim(X) + 1$;
this proves the result because $\cD^0 = \Db(X)$.
For the induction base 
$\cD^{\dim(X)+1}$ 
consists of the zero-complex, so the statement holds.
Assume $\cD^{m+1} \subset \cT$, 
for some $m \ge 0$.
To show that $\cD^{m} \subset \cT$,
by Lemma \ref{lem:devissage},
it suffices to check that 
for all
integral subschemes $Z \subset X$
of codimension $m$ we have $\cO_Z \in \cT$.
Let $x \in Z$ be the generic point.
By Lemma \ref{lem:FZ},
applied to $\pi \sigma$,
$\cF_Z^\bullet$
is a complex supported on 
a proper closed subset of $Z$ and the statement 
follows by the induction hypothesis.
\end{proof}

\begin{proof}[Proof of Theorem \ref{thm:main-intro}]
By Proposition \ref{prop:rational-acyclic},
$\pi$ is $\cO$-acyclic
after a modification, hence
by Proposition \ref{prop:acylic-surjective}, 
the image of $\bR\pi_*$ generates $\Db(X)$.
\end{proof}

\begin{proof}[Proof of Corollary \ref{cor:cat-contraction}]
Let $\cT \subseteq \Db(X)$
be the image of $\ol{\bR\pi_*}$. 
Then $\cT$ is closed under shifts,
and since $\ol{\bR\pi_*}$
is assumed to be full,
$\cT$ is also closed under taking cones
of morphisms. 
Thus $\cT$ is a
triangulated subcategory
of $\Db(X)$.
On the other hand, by
Theorem \ref{thm:main-intro},
$\cT$ generates $\Db(X)$
as a triangulated category.
Since $\cT$ is already 
closed under taking 
shifts and cones, we obtain $\cT = \Db(X)$.
\end{proof}

\section{Examples and counterexamples}

Both Bondal--Orlov Conjecture
\ref{conj:BO} and Theorem \ref{thm:main-intro}
fail when $X$ does not have rational singularities.

\begin{example}\label{ex:non-rat}
If $\kk$ is a non-closed field, 
$C \subset \P^2$ 
is a smooth 
cubic curve without rational points, $X \subset \bA^3$ is a cone over $C$, and $\tX$ is the blow-up of the vertex of the cone $X$. The singularity is not rational and 
$\bR\pi_*\colon \rG_0(\tX) \to \rG_0(X)$
is not surjective because the image
does not contain $[\cO_P]$, essentially because
the Euler characteristic of every object $\cE \in \Db(C)$ is divisible
by $3$, hence never equals one.
\end{example}

In Theorem \ref{thm:main-intro},
we cannot replace the Grothendieck group $G_0(X)$ with either integral Chow groups or Borel--Moore homology.
We explain this in detail.

For Chow groups with rational coefficients,
the induced
map $\pi_* \colon \CH_*(\tilde{X}, \QQ) \to \CH_*(X, \QQ)$ is always surjective
for a proper morphism $\pi\colon \tX \to X$
since any closed subset $Z \subseteq X$ is dominated by a closed subset $\tilde{Z} \subseteq \tilde{X}$ such that the restriction $\pi|_{\tilde{Z}} \colon \tilde{Z} \to Z$ is generically finite. However, integrally this is not always the case.

\begin{example}\label{rem:homology1}
If $\kk$ is a non-closed field, $X \subset \bA^3$ is a cone over a conic $C$ without rational points, and $\tX$ is the blow-up of the vertex of the cone $X$,
then $\pi_*\colon \CH_0(\tX)=\CH_0(C) \to \CH_0(X)=\ZZ$
is not surjective, since the pushforward of any 0-cycle in $C$ has even degree. 
\end{example}

Examples
analogous to Examples
\ref{ex:non-rat}
and
\ref{rem:homology1}
can be constructed
over $\kk = \CC$
by spreading out.
We now consider pushforwards
for Borel--Moore homology for complex varieties,
which for proper varieties
coincides with the usual homology.

\begin{example}\label{rem:homology2}
Let $X$ be a projective complex
threefold with only isolated nodal (thus rational) singularities, and $\pi\colon \tilde{X} \to X$  be the blow-up of the nodes $\Sigma$ of $X$ with exceptional divisor $E_p \simeq \bP^1 \times \bP^1$ for $p \in \Sigma$.
The pushforward $\rH_3(\tX, \QQ) \to \rH_3(X, \QQ)$ may not be surjective.

To explain this, we need to introduce some notation.
Let $\delta(X) := \rank(\Cl(X)/\Pic(X))$ be the
\emph{defect} of $X$. We have $0 \le \delta(X) \le |\Sigma|$, see e.g.\ \cite[Corollary 3.8]{Kalck-Pavic-Shinder}.
Furthermore, if $X$ is a nodal hypersurface
in $\P^4$ of degree $d \ge 3$, or a nodal double cover of $\P^3$ branched in a surface of degree $d \ge 4$,
then $\delta(X) < |\Sigma|$, see \cite[Example  3.13]{Kalck-Pavic-Shinder}.
The defect $\delta(X)$ can be computed from the resolution $\tX$
via 
\[
|\Sigma| - \delta(X) = 
\rank \Coker\left(\Pic(\tX) \to  \bigoplus_{p \in \Sigma} \Pic(E_p)\right) 
=
\rank \Coker\left(H^2(\tX) \to \bigoplus_{p \in \Sigma} \rH^2(E_p)\right),
\]
where cohomology are taken with integral or rational coefficients. 
Using duality we can also write
\[
|\Sigma| - \delta(X) = 
\rank \Ker\left(\bigoplus_{p \in \Sigma} \rH_2(E_p) \to \rH_2(\tX)\right).
\]
The Mayer--Vietoris exact sequence for the mapping cylinder of $\pi$ reads
\[ \rH_3(\tilde{X}) \xrightarrow{\pi_*} \rH_3(X) \to \bigoplus_{p \in \Sigma} \rH_2(E_p) 
\to \rH_2(\tilde{X}),\]
thus $\pi_* \colon \rH_3(\tilde{X}) \to \rH_3(X)$ is not surjective provided $\delta(X) < |\Sigma|$.
\end{example}

We may restrict to
the pure part of the homology $\rH^{\mathrm{BM}}_{\mathrm{pure}, i}(X, \QQ)$, i.e.\ the lower piece of the weight filtration of Deligne's mixed Hodge structure on $\rH^{\mathrm{BM}}_{i}(X, \QQ)$. If $\pi \colon \tilde{X} \to X$ is a resolution of singularities of a proper normal complex
variety $X$, 
then $\pi_* \colon \rH_{i}(\tilde{X}, \QQ) \to \rH_{\mathrm{pure}, i}(X, \QQ)$ is surjective, independently of the singularities of $X$; 
see \cite[Theorem 5.41]{PS08}.
\vspace{0.5 cm}\\

\corr{
\subsection*{Conflicts of interest.} none.

\subsection*{Financial support} M.M. was supported by the Institute of Science and Technology Austria. This project has received funding from the European Union’s Horizon 2020 research and innovation
programme under the Marie Skłodowska-Curie grant agreement No 101034413.

E.S. was partially supported by the EPSRC grant EP/T019379/1 ``Derived categories and algebraic K-theory of singularities'', and by the ERC Synergy grant ``Modern Aspects of Geometry: Categories, Cycles and Cohomology of Hyperkähler Varieties".}

\bibliography{fano}
\bibliographystyle{alpha}

\end{document}